\documentclass[preprint,12pt]{elsarticle}

\usepackage{graphics}

\usepackage{amssymb}
\usepackage{lineno}
\usepackage{amsthm}
\usepackage{amsmath}
\usepackage{mathabx}
\usepackage{color}
\usepackage{amsopn}
\usepackage{amsmath}
\usepackage{mathabx}
\usepackage{pstricks}
\usepackage{color}
\usepackage{pst-plot}
\usepackage{pst-tree}
\usepackage{pst-eps}
\usepackage{multido}
\usepackage{pst-node}
\usepackage{pst-3d}
\usepackage{colordvi}
\usepackage{multirow}
\usepackage{pst-grad}
\usepackage{colortbl}

\DeclareMathOperator{\interior}{Int}

\newtheorem{theorem}{Theorem}[section]

\newtheorem{definition}[theorem]{Definition}

\newtheorem{lemma}[theorem]{Lemma}
\newtheorem{proposition}[theorem]{Proposition}
\newtheorem{remark}[theorem]{Remark}

\journal{ }

\begin{document}

	\renewcommand{\listtablename}{\'indice de tablas}

	\nocite{*}
	
	\begin{frontmatter}
		
		
		
\title{Extent of occurrence reconstruction using a new data-driven support estimator}
\author[rvt]{A. Rodr\'iguez-Casal}
\author[rvt]{P. Saavedra-Nieves \corref{cor1}}
\address[rvt]{Department of Statistics, Mathematical Analysis and Optimization, Universidade de Santiago de Compostela, Spain}
\cortext[cor1]{Corresponding author: paula.saavedra@usc.es (P. Saavedra-Nieves)}

\begin{abstract}
Given a random sample of points from some unknown distribution, we propose a new data-driven method for estimating its probability support $S$. Under the mild assumption that $S$ is $r-$convex, the smallest
$r-$convex set which contains the sample points is the natural estimator. The main problem for using this estimator in practice
is that $r$ is an unknown geometric characteristic of the set $S$. A stochastic algorithm is proposed for determining an optimal estimate of $r$ from the data under mild regularity assumptions on the density function. The resulting data-driven reconstruction of $S$ attains the same convergence rates as the convex hull for estimating convex sets, but under a much more flexible smoothness shape condition. The new support estimator will be used for reconstructing the extent of occurrence of an assemblage of invasive plant species in the Azores archipelago.\end{abstract}

%
%


\end{frontmatter}

\section{Introduction}
\label{intro}

Natural reserve network designs require information about species occurrence data. One of the most widely handled concepts is the extent of occurrence (EOO). In fact, the International Union for the Conservation of Nature (IUCN) establishes the EOO as a key measure of extinction risk. Roughly speaking, the IUCN defines the EOO as the area contained within the shortest continuous imaginary boundary which can be drawn to encompass all the known, inferred or projected sites of present occurrence of a taxon, excluding cases of vagrancy. For a complete review on this subject, see Rondinini et al. (2006). 

The problem of EOO reconstruction will be illustrated via the analysis of a real dataset containing $740$ geographical coordinates (or occurrences) for $28$ species of terrestrial invasive plants distributed in two of the Azorean islands (Terceira and S\~ao Miguel) from 2010 until 2018. In Figure \ref{fig1}, a satellite image of major Azorean islands (top, left) and five of the invasive species are shown (bottom). The $740$ geographical locations (slightly jittered) are represented on the map of Terceira and S\~ao Miguel islands in Figure \ref{fig1} (top, right). This dataset is available from the Global Biodiversity Information Facility (GBIF) website (see GBIF.org, 27th May 2019).

An initial estimation of the EOO for this assemblage of invasive plants was obtained from GeoCAT. It is an open source, browser based tool endorsed by IUCN that allows to reconstruct the EOO from the geographical locations of species or taxon. Users can quickly combine data from multiple sources including GBIF datasets which can be easily imported. The GeoCAT reconstruction of the EOO for the assemblage of plant species used here as an example is given by the convex hull of the sample of the $740$ coordinates, $H(\mathcal{X}_{740})$. Mathematically, $H(\mathcal{X}_{740})$ is the smallest convex set that contains $\mathcal{X}_{740}$. In fact, it is computed as the intersection of all half spaces containing $\mathcal{X}_{740}$. For more details, compare Figure \ref{fig2} (first row, left) and Figure \ref{fig2} (second row, left). Note that this EOO estimation presents some limitations because a marine area is inside the $H(\mathcal{X}_{740})$. Obviously, none of the plant species considered here can occur in open sea which should remain outside the EOO. Therefore, convexity can be a too restrictive shape condition to be assumed in practice.

Our goal is to propose a more realistic and automatic EOO reconstruction from support estimation perspective. This methodological approach has proved to be useful in different disciplines such as image analysis (see Rodr\'iguez-Casal and Saavedra-Nieves, 2016), quality control (see Devroye and Wise, 1980 or Chevalier, 1976) or animals home range estimation (see De Haan and Resnick, 1994 or Ba\'illo and Chac\'on, 2018). However, the problem of reconstructing the EOO has not been yet considered formally under this viewpoint.

In general, support estimation deals with the problem of reconstructing the
compact and nonempty support $S\subset \mathbb{R}^d$ of an absolutely continuous random vector $X$ from a
random sample $\mathcal{X}_{n}=\{X_{1},...,X_{n}\}$ (see Cuevas and Fraiman, 2010 for a complete survey on the subject). Of course, when the support $S$ is assumed to be convex then the convex hull of the sample points, $H(\mathcal{X}_n)$, provides a natural support estimator. See Schneider (1988, 1993), D\"{u}mbgen and Walther (1996) or Reitzner (2003), for thorough analysis of this estimator. This estimator is indeed simple, but it may not be suitable for practical situations, failing to provide a satisfactory support estimator when $S$ is disconnected as in the example of invasive plants in Azores archipelago where the occurrences are distributed in two different islands.

 \begin{figure}
 	\hspace{-1.5cm}\includegraphics[scale=1]{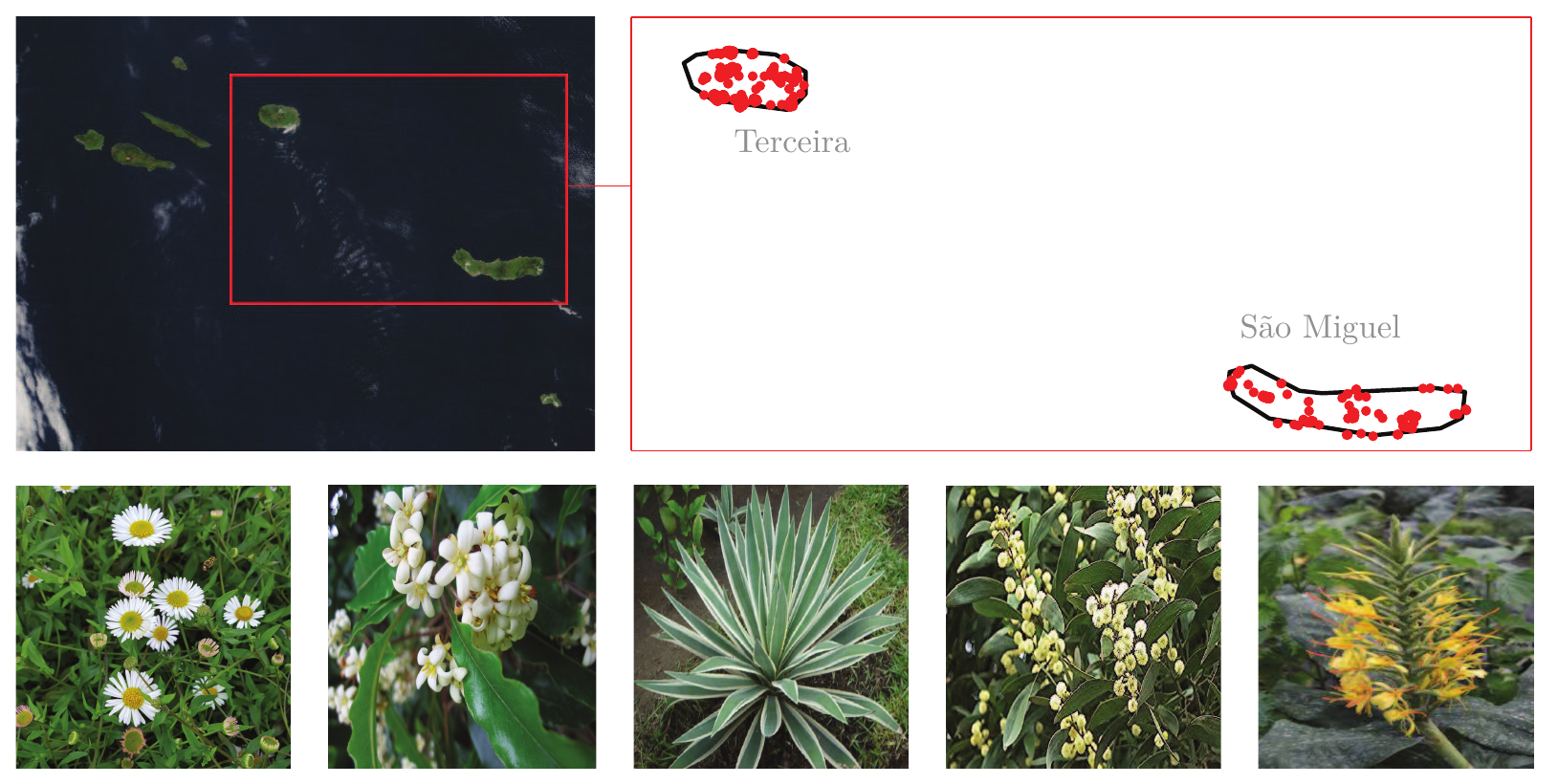}
 	\caption{Location of Terceira and S\~ao Miguel islands in the Azores Archipelago, NASA satellite image (top, left). The enlarged area (top, right) shows the 740 geographical locations used to reconstruct the EOO of an assemblage of 28 invasive plant species including: Erigeron karvinskianus, Pittosporum undulatum, Agave americana, Acacia melanoxylon, Hedychium gardnerianum (bottom, from left to right).}\label{fig1}
 \end{figure}

In this work, we will propose a new data-driven support estimator and, as a consequence, an original and realistic EOO reconstruction that will overcome the limitations derived from convexity restriction. Concretely, we assume that the support $S$ satifies the $r-$convexity shape condition for $r>0$, a much more flexible property than convexity as it will be shown. Our proposal considers the smallest $r$-convex set containing $\mathcal{X}_{n}$ ($r-$convex hull of $\mathcal{X}_{n}$, namely $C_r(\mathcal{X}_{n})$) as the natural estimator for the usually unknown support. This estimator is well known in the computational geometry literature for providing
reasonable global reconstructions if the sample points are (approximately) uniformly distributed on the
set $S$ (see Edelsbrunner, 2014). In fact, despite being
 $r-$convexity a more general condition than convexity, $C_r(\mathcal{X}_n)$ can achieve the same convergence rates than $H(\mathcal{X}_n)$ as proved by Rodr\'{\i}guez-Casal (2007). However, this
 estimator presents an important disadvantage: it depends on the commonly unknown parameter $r$. Although the influence of $r$ is considerable, it must be specified by the practitioner (see Joppa et al., 2016). For the example of invasive species in Azorean islands,  Figure \ref{fig2} shows $C_r(\mathcal{X}_{740})$ for different values of $r$. Small values of $r$ provide fragmented estimators (many isotated points and connected components) leading to an EOO reconstruction which resembles $\mathcal{X}_n$ (Figure \ref{fig2}: second row, right). If $r=0.3$, a realistic reconstruction of the EOO is obtained since sea areas are not inside the estimator (Figure \ref{fig2}: third row, left). However,
 if large values of $r$ are considered then $C_r(\mathcal{X}_n)$ basically coincides with $H(\mathcal{X}_n)$ (Figure \ref{fig2}: third row, right). Therefore, arbitrary choices of $r$ may provide incongruous EOO estimations.
 
Most of the available results in the literature about support estimation make special emphasis on asymptotic properties, especially
consistency and convergence rates but they do not usually give any criterion for selecting the unknown parameter $r$ in $C_r(\mathcal{X}_n)$ from the sample. The aim of this paper is to overcome this drawback and present a method for selecting the parameter $r$ for the $r-$convex hull estimator from the available data. This problem has scarcely been studied in the statistical literature with just a couple of references available on the topic. First, Mandal and Murthy (1997) proposed a selector for $r$ based on
the concept of minimum spanning tree but only consistency of the method was provided without considering optimality issues. Later, Rodr\'iguez-Casal and Saavedra-Nieves (2016) proposed an automatic selection criterion based on a very intuitive idea for the selection of $r$ but under the restriction that the sample distribution is uniform. According to Figure \ref{fig2} (bottom, right), sea areas are contained
in $C_r(\mathcal{X}_n)$ if the selected $r$ is too large. So, the estimator contains a large ball empty of sample points, see gray balls in Figure \ref{fig2} (top, left) and (bottom, right). Janson (1987) calibrated the size of this maximal ball (or spacing)
when the sample distribution is uniform on $S$. Berrendero et al. (2012) used this result to test uniformity when the support is unknown. However, Rodr\'iguez-Casal and Saavedra-Nieves (2016) followed the somewhat opposite approach. They assume that $\mathcal{X}_n$ comes from a uniform distribution on $S$ and if a big
enough spacing is found in $C_{r}(\mathcal{X}_n)$ then it is incompatible with the assumption that data are uniform. As a consequence, it is concluded that $r$ is too large. Therefore, it is proposed to select the largest value of $r$ compatible with the uniformity assumption on $C_{r}(\mathcal{X}_n)$. 

Recently, Aaron et al. (2017) extended the results by Janson (1987) to the case where the data are generated from
a density $f$ that is bounded from below and Lipschitz continuous restricted to its bounded support. Here, we will use this extension in order to derive a test to decide, given a fixed $r>0$, whether the unknown support $S$ is $r-$convex with no more information apart from $\mathcal{X}_n$. In this case, if a large
enough spacing is found in $C_{r}(\mathcal{X}_n)$ then the null 
hypothesis of $r-$convexity will be rejected. A new data-driven selector for the index $r$ will be established from this test. Following the scheme in  Rodr\'iguez-Casal and Saavedra-Nieves (2016), it is proposed to choose the largest value of $r$ compatible with the $r-$convexity assumption.

Once the parameter $r$ is estimated from $\mathcal{X}_n$, a new data-driven support reconstruction, based on the estimator of $r$, will be proposed. As a consequence, a flexible reconstruction for the EOO will be obtained.

This paper is organized as follows. Mathematical tools are introduced in Section \ref{nosometedo}. First, the geometric assumptions on $S$ and the optimal value of the parameter $r$ to be estimated are introduced. Then, the regularity assumptions on $f$ and a new nonparametric estimator are established. At last, the maximal spacing and its estimator are formally defined. In Section \ref{testrconvex}, we propose a procedure for testing the null 
hypothesis that $S$ is $r-$convex for a given $r>0$. This test will play a key role in the definition of the consistent estimator of $r$. Then, a new estimator for the support $S$ is proposed in Section \ref{cons} and it will be seen that it achieves the same convergence rates as the convex hull for estimating convex sets. The main
numerical features involving the practical application of the algorithm are exposed in Section \ref{numerical}. In Section \ref{reald}, the performance of the new support reconstruction will be analyzed estimating the EOO of an assemblage of terrestrial plant species in two Azorean islands. Conclusions are exposed in Section \ref{con}. In Section \ref{proof}, we detail the proofs of theoretical results. Finally, some auxiliary results are deferred to Section \ref{aux}.


\section{Mathemathical tools}\label{nosometedo}

Regularity conditions, namely shape assumptions on $S$, will be introduced next. In addition, we will discuss which is the optimal value of the shape index $r$ to be estimated. Then, required conditions on the density function $f$ and an original nonparemetric kernel estimator will be also presented. Finally, basic notions on maximal spacings are established.

\subsection{About geometric assumptions on $S$ and the optimal value of $r$}
 In this work, $S$ is assumed to be $r-$convex for some $r>0$. Therefore, it is necessary to establish the formal definition of this geometric property in Definition \ref{rconvexi}.

\begin{definition}\label{rconvexi}
 A closed set $A\subset\mathbb{R}^d$ is said to be $r-$convex, for some $r>0$, if $A=C_{r}(A)$, where
$$C_{r}(A)=\bigcap_{\{B_r(x):B_r(x)\cap
	A=\emptyset\}}\left(B_r(x)\right)^c$$
denotes the $r-$convex hull of $A$ and $B_r(x)$, the open ball with
center $x$ and radius $r$. 
\end{definition} 

In practice, $C_{r}(\mathcal{X}_n)$ can be computed as the intersection of the complements of all open balls of radius larger than or equal to $r$ that do not intersect $\mathcal{X}_n$. In Figure \ref{fig111}, the computation of $C_{r}(\mathcal{X}_{740})$ for $r=0.3$ (left) and $r=5$ (right) is shown considering the example in Azorean islands. Note that $C_{0.3}(\mathcal{X}_{740})$ is an acceptable EOO reconstruction equal to the intersection of the complements of all gray open balls represented. However, if we select $r=5$, marine areas are clearly inside the $C_{5}(\mathcal{X}_{740})$.

Furthermore, the concept of $r-$convex hull is closely related to the closing of $A$ by $B_r(0)$ from the mathematical morphology, see Serra (1982). It can be shown that
$$
C_{r}(A)=(A\oplus r B)\ominus r B,
$$
where $B=B_1(0)$, $\lambda C=\{\lambda c: c\in C\}$, $C\oplus D=\{c+d:\ c\in C, d\in D\}$ and $C\ominus D=\{x\in\mathbb{R}^d:\ \{x\}\oplus D\subset C\}$, for $\lambda \in \mathbb{R}$ and
sets $C$ and $D$.

 \newpage

\begin{figure}   
\hspace{-1.3cm}	\includegraphics[scale=1]{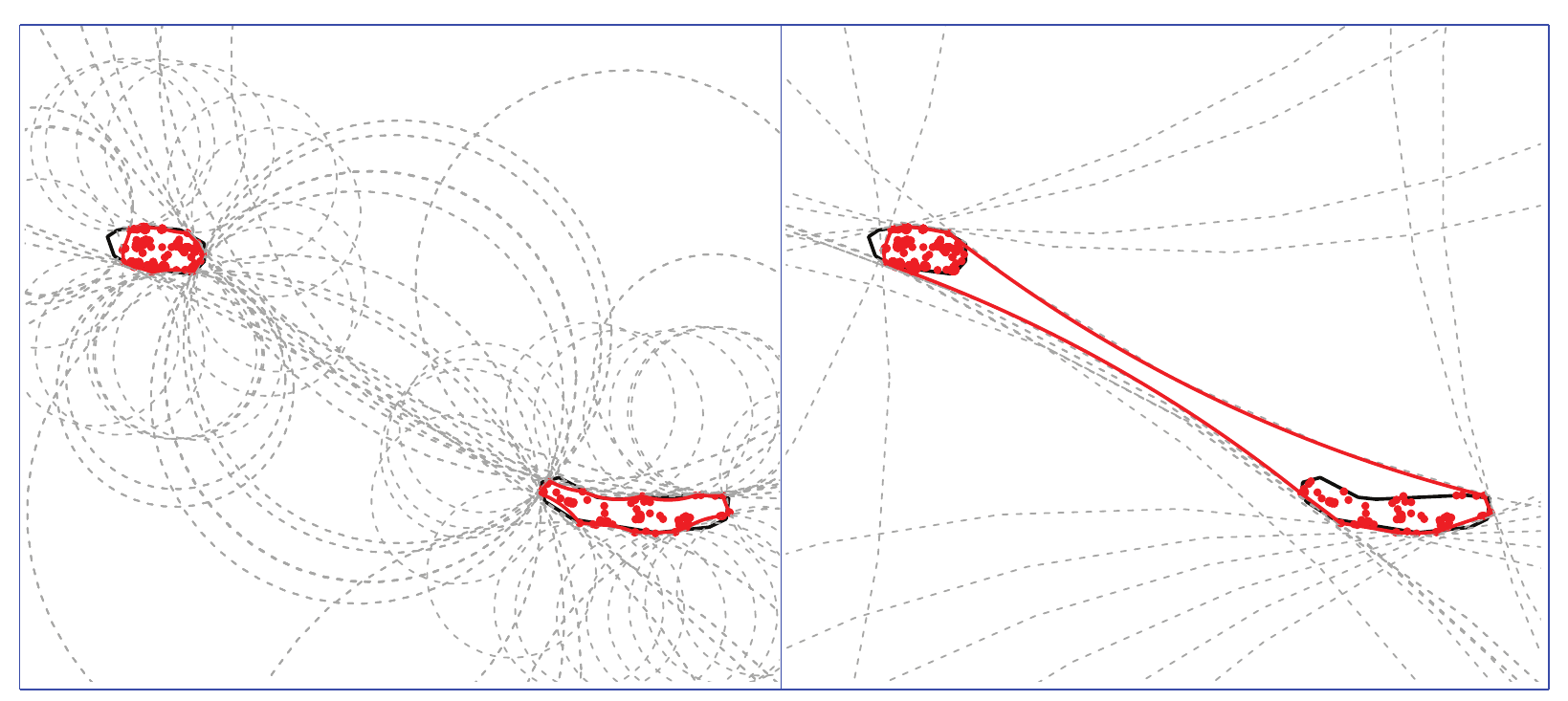}
	\vspace{.5cm}\caption{$C_{r}(\mathcal{X}_{740})$ (red color) and $B_{r^*}(x)$ for $r^*\geq r$ (gray color) such that $B_{r^*}(x)\cap
		\mathcal{X}_{740}=\emptyset$ taking $r=0.3$ (left) and $r=5$ (right).}\label{fig111}
\end{figure}

As it has been mentioned in the Introduction, the problem of reconstructing a $r-$convex support $S$ using a data-driven procedure could be easily solved if the parameter $r$ is estimated from a
random sample of points $\mathcal{X}_n$ taken in $S$. The first step is to determine precisely
the optimal value of $r$ to be estimated, which is established in Definition \ref{sup}: we propose to estimate the largest value of $r$ which verifies that $S$ is $r-$convex.

\begin{definition}\label{sup}Let $S\subset \mathbb{R}^{d}$ a compact, nonconvex and $r-$convex set for some $r>0$. It is defined
\begin{equation}\label{maximo2}
r_0=\sup\{\gamma>0:C_\gamma(S)=S\}.\end{equation}
\end{definition}

For simplicity, it is assumed that $S$ is not convex (of course, if $S$ is convex $r_0$ would be infinity). Proposition 2.4 in Rodr\'iguez-Casal and Saavedra-Nieves (2016) shows that, under mild regularity conditions, the supreme established in (\ref{maximo2}) is a maximum, that is, $S$ is $r_0-$convex and $r-$convex for all $r< r_0$. Under this hypothesis, the optimality of the smoothing parameter defined in (\ref{maximo2}) can be justified. It is clear that $S$ is $r-$convex for $r\leq r_0$ but if  $r<r_0$, $C_{r}(\mathcal{X}_n)$ is a non admisible estimator since it is always outperformed by $C_{r_0}(\mathcal{X}_n)$. This happens because, with probability one,
$C_{r}(\mathcal{X}_n)\subset C_{r_0}(\mathcal{X}_n)\subset S$. It should also noted that, for $r>r_0$, even for $r$ very close to $r_0$, $C_{r}(\mathcal{X}_n)$ would considerably overestimate $S$. For instance, if $S$ is equal to the circular ring in Figure \ref{Figura1} (right) and $r>r_0$, $C_{r}(S)$ coincides with the outer circle. The mild regularity condition we need is slightly stronger than $r-$convexity: \vspace{4mm}\\
($R\label{new}$) $S$ fulfills the $r-$rolling property and $S^c$ fulfills the $\lambda-$rolling condition for some $r$, $\lambda>0$.\vspace{3mm}\\
Following Cuevas et al. (2012), it is said $A$ satisfies the (outside) $r-$rolling condition if each boundary point $a\in\partial A$ is contained
in a closed ball with radius $r$ whose interior does not meet $A$. There exist interesting relationships between this
property and $r-$convexity. In particular, Cuevas et al. (2012) proved that if $A$ is compact and $r-$convex then $A$ fulfills the $r-$rolling condition. According to Figure \ref{Figura1} (left), the reciprocal is not always true. 
Proposition 2.2 in Rodr\'iguez-Casal and Saavedra-Nieves (2016) shows that ($R$) is a (mild) sufficient condition to ensure the $r-$rolling condition implies
$r-$convexity. 
Condition ($R$) was essentially analyzed by Walther (1997, 1999) but just
the case $r=\lambda$ was taken into account. In this work, the radius $\lambda$ can be different from $r$, see
Figure \ref{Figura1} (center). Walther (1997, 1999) proved that, if $r=\lambda$, $\partial S$ is a $\mathcal{C}^1$ $(d-1)-$dimensional submanifold in $\mathbb{R}^d$ and that $S$ is $r-$convex. Proposition 2.2 in Rodr\'iguez-Casal and Saavedra-Nieves (2016) generalized this property since, for $\lambda<r$, Walther's result would only imply $\lambda$-convexity but not $r-$convexity. So, for sets
satisfying  ($R$),  $r-$convexity is ensured,
even for very small values of $\lambda$.



Proposition 2.2 in Rodr\'iguez-Casal and Saavedra-Nieves (2016) is the key for proving that $r_0$ is a maximum. To see this, let be $\{r_n\}$ a sequence converging
to $r_0$ such that $C_{r_n}(S)=S$. This sequence always exists by Definition \ref{sup}. It can be proved, using the results by Cuevas et al. (2012),
that $S$ satisfies the $r_n-$rolling condition and, by Proposition 2.3 in Rodr\'iguez-Casal and Saavedra-Nieves (2016), this property is preserved in the limit, so $S$ is also $r_0$-rolling. Finally, under ($R$), $r_0$-rolling implies that $S$ is $r_0-$convex.


The authors conjecture that the equivalence between $r-$convexity and $r-$rolling could be stated in a more general framework and it may be proved under milder conditions.

\begin{figure}[h!]
	\hspace{-1.2cm} \includegraphics[scale=1]{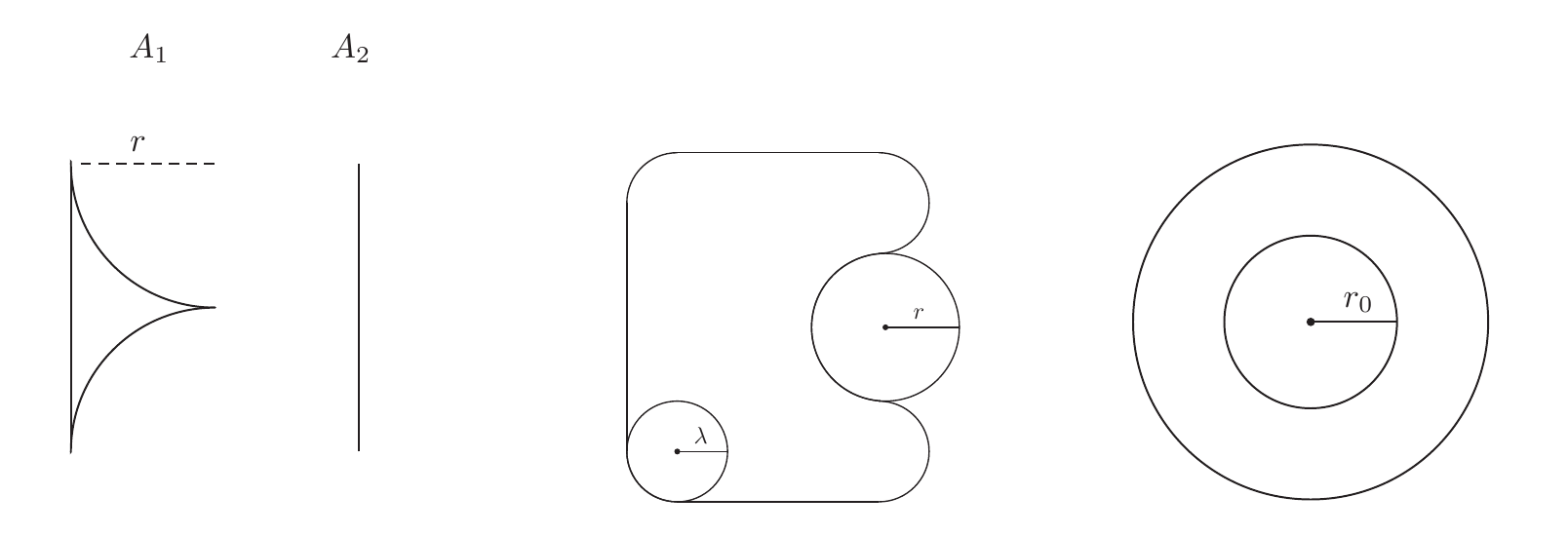}
	
	\caption{$A_1\cup A_2$ fulfills the $r-$rolling condition $\nRightarrow$ $A_1\cup A_2$ is $r-$convex (left). ($R$) is a more general condition (center). Circular ring with inner circle of radius $r_0$ (right).}\label{Figura1}
\end{figure}

\begin{remark}Under certain conditions of $S$ (for instance, $\interior(H(S))\neq \emptyset$), it is verified that $C_{\infty}(S)=H(S)$ where $C_{\infty}(S)=\lim_{r_n\rightarrow\infty}C_{r_n}(S)$. Therefore, if $S$ is assumed to be convex, Proposition 2.4 in Rodr\'iguez-Casal and Saavedra-Nieves (2016) remains true. For more details, see Walther (1999).\end{remark}

\subsection{About regularity conditions on $f$ and its nonparametric estimation}\label{rc}

All through this paper, we assume that the random sample of points, $\mathcal{X}_n$, is generated from a density $f$ that satisfies the next regularity condition:\vspace{4mm}\\($f_{0,1}^L\label{newf}$) The restriction of the density $f$ to $S$ is Lipschitz continuous (there exists $k_f$ such that $\forall x,y\in S$,\\ \textcolor{white}{($f_{0,1}^L$)} $|f(x)-f(y)|\leq k_f\|x-y\|$ and there exists $f_0>0$ such that $f(x)\geq f_0$ for all $x\in S$. Furthermore,\\\textcolor{white}{($f_{0,1}^L$)} $f_1=\max_{x\in S} f(x)$.\vspace{2mm}\\

 As an example in the one-dimensional case, condition ($f_{0,1}^L$) is satisfied by $f(x)=1/(b-a)$ if $x\in[a,b]$ and $f(x)=0$, otherwise where $a$ and $b$ denote two real numbers verifying that $a<b$. \vspace{.05mm}\\

Morever, a non-conventional density estimator will be introduced in Definition \ref{fn}.

\begin{definition}\label{fn} Let $r>0$ and let $Vor(X_i)$ be the Voronoi cell of the point $X_i$ (i.e. $Vor(X_i)=\{x:\|x-X_i\|=\min_{y\in\mathcal{X}_n}\|x-y\|\}$). If $K$ is a kernel function and $f_n(x)=\frac{1}{nh_n^d}\sum K((x-X_i)/h_n)$ denotes the usual kernel density estimator, we define$$\hat{f}_n(x)=\max_{i:x\in Vor(X_i)}f_n(X_i)\mathbb{I}_{x\in C_{r}(\mathcal{X}_n)}.$$\end{definition}

Note that this nonparametric estimator have a non-usual behaviour: it is expected to converge towards the unknown density when the support is $r-$convex, but not when the support is not $r-$convex.

Moreover, some technical hypotheses on the kernel function must be established. \vspace{4mm}\\($\mathcal{K}_{\phi}^p$) The kernel function $K$ belongs to the set of kernels $\mathcal{K}$ such that $K(u)=\phi(p(u))$ where $p$ is a polynomial \textcolor{white}{($\mathcal{K}_{\phi}^p$)} and $\phi$ is a is bounded real function of bounded variation, verifying that $c_K=\int \|u\|K(u)du<\infty$, $K\geq 0$ \textcolor{white}{($\mathcal{K}_{\phi}^p$)} and there exists $r_K$ and $c^{'}_K>0$ such that $K(x)\geq c^{'}_K$ for all $x\in B_{r_K}[0]$.\vspace{2.3mm}\\

Condition ($\mathcal{K}_{\phi}^p$) is satisfied, for instance, by the Gaussian kernel.
\subsection{About maximal spacings and its nonparametric estimation}\label{rc2}
The optimal value of the shape index $r$ to be estimated is just established in Definition \ref{sup}. Some concepts on maximal spacings theory must be handled to propose a consistent estimate of $r$. 

The notion of maximal-spacing in several dimensions was introduced and studied
by Deheuvels (1983) for uniformly distributed data on the unit cube. Later on,
Janson (1987) extended these results to uniformly distributed data on any bounded
set and derived the asymptotic distribution
of different maximal-spacings notions without conditions on the shape of the support $S$. Aaron et al. (2017) generalized the results by Janson (1987) to the non-uniform case.  

The shape of the considered spacings will be defined by a
given set $A\subset \mathbb{R}^d$. For the validity of the theoretical results, it is sufficient to assume that $A$ is a compact and convex set. For practical purposes, the usual choices
are $A = [0, 1]^d$ or $A = B_1[0]$, the closed ball of center $0$ and radius $1$. For a general dimension $d$, the first definition of maximal spacing is that
used by Janson (1987) under the restriction of data are uniformly distributed:
$$\Delta_n^*(\mathcal{X}_n)=\sup\{\gamma:\exists x\mbox{ such that }\{x\}+\gamma A \subset S\setminus \mathcal{X}_n\}.$$
If the Lebesgue measure of the set $A$ is one, $\Delta_n^*(\mathcal{X}_n)^d$ represents the Lebesgue measure of the largest set $\{x\}+\gamma A \subset S\setminus \mathcal{X}_n$. The concept of maximal spacing can be related easily to the maximal inner radius when $A=B_1[0]$. If $\interior(S)\neq \emptyset$, the maximal inner radius of $S$ is defined as
$$\mathcal{R}(S)=\sup\{\gamma>0:\exists x \in S \mbox{ such that }B_{\gamma}[x]\subset S\}.$$Note that the value of the maximal spacing depends on $S$ and also on $\mathcal{X}_n$. However, the definition of the maximal inner radius relies only on $S$.

Aaron et al. (2017) extended the
definition of maximal-spacing assuming that $\mathcal{X}_n$ is drawn according to a density $f$ with bounded support $S$, the Lebesgue measure of the set $A$ is one and its barycentre is the origin of $\mathbb{R}^d$. In this more general setting, the maximal spacing is defined as
$$\Delta_n(\mathcal{X}_n)=\sup\left\{\gamma:\exists x\mbox{ such that }\{x\}+\frac{\gamma}{f(x)^{1/d}}A\subset S\setminus \mathcal{X}_n\right\}$$and
$$V_n(\mathcal{X}_n)=\Delta_n(\mathcal{X}_n)^d.$$The previous definition of maximal spacing relies on density $f$. In this way, it distinguishes between low and high density regions. Throughout this paper, we will assume this latter choice $A =  w_d^{-1/d}B_1[0]$ where $w_d$ denotes the Lebesgue measure of $B_1[0]$.

Janson (1987) calibrated the volume of the maximal spacing under uniformity assumptions without conditions on the shape of the support $S$. The corresponding extension established in Theorem 2 in Aaron et al. (2017) is shown in Theorem \ref{aaronetal} modifying slightly the original hypotheses on $f$ and on the shape of $S$. The result remains true if it is assumed that $S$ is under ($R$) and the density function $f$ satifies ($f_{0,1}^L$).

\begin{theorem}\label{aaronetal} 
Let $\mathcal{X}_n$ be a random and i.i.d sample drawn according to a density $f$ that satisfies ($f_{0,1}^L$) with compact and nonempty support $S$ under ($R$).
Let $U$ be a random variable with distribution
$$
\mathbb{P}(U\leq u) =\exp(-\exp(-u)) \mbox{ for } u\in \mathbb{R}
$$ and let $\beta$ be a constant specified in Janson (1987). Then, we have that
$$U(\mathcal{X}_n)\stackrel{d}{\rightarrow} U\mbox{ when }n\rightarrow\infty,$$
$$\liminf_{n\rightarrow\infty}\frac{nV_n(\mathcal{X}_n)-log(n)}{log(log(n))}\geq d-1\mbox{ a.s.}, \mbox{ }\limsup_{n\rightarrow\infty}\frac{nV_n(\mathcal{X}_n)-log(n)}{log(log(n))}\leq d+1\mbox{ a.s.}$$where
$$U(\mathcal{X}_n)=nV_n(\mathcal{X}_n)-log(n)-(d-1)log(log(n))-log(\beta).$$ 
\end{theorem}

\begin{remark}The value of constant $\beta$ does not depend on $S$. It is explicitly given in Janson (1987). Concretely,
$$\beta=\frac{1}{d!}\left(\frac{\sqrt{\pi}\Gamma\left(\frac{d}{2}+1\right)}{\Gamma\left(\frac{d+1}{2}\right)}\right)^{d-1}.$$ In particular, for the bidimensional case, $\beta=1$.
\end{remark}

A plug-in estimator of the maximal spacing $\Delta_n(\mathcal{X}_n)$ will be proposed next. Note that the definition of $\Delta_n(\mathcal{X}_n)$ relies on the support $S$ and also on the density function $f$ (both are usually unknown). Under the assumption of $r-$convexity, $S$ will be estimated as $C_{r}(\mathcal{X}_n)$. As for the density function $f$, the new nonparametric density estimator introduced in Definition \ref{fn} will be used. Then, we define the following plug-in estimator of $\Delta_n(\mathcal{X}_n)$:
$$\hat{\delta}(C_{r}(\mathcal{X}_n)\setminus\mathcal{X}_n)=\sup\left\{\gamma:\exists x\mbox{ such that }\{x\}+\frac{\gamma }{\hat{f}_n(x)^{1/d}}A\subset C_{r}(\mathcal{X}_n)\setminus \mathcal{X}_n\right\}.$$
Note that if $S$ is $r-$convex, $\hat{\delta}(C_{r}(\mathcal{X}_n)\setminus\mathcal{X}_n)$ should converge to zero as the sample size increases. However, if $S\subsetneq C_r(S)$, the plug-in estimator of $\Delta_n(\mathcal{X}_n)$ is expected to converge to a positive constant.

\section{A new test for $r-$convexity}\label{testrconvex}

We will introduce a consistent hypothesis test based on $\mathcal{X}_n$ drawn according to an unknown density $f$ on the unknown support $S$, to asses $r-$convexity for a certain $r>0$. This test is crucial for defining an estimator of $r_0$ that would allow the data-driven estimation of the support $S$.

Given $r>0$, the null hypothesis that $S$ is $r-$convex will be tested taking the volume of $\hat{\delta}(C_{r}(\mathcal{X}_n)\setminus\mathcal{X}_n)$ as statistic. The idea that supports this procedure is simple:
Under ($f_{0,1}^L$) and ($R$), Theorem \ref{aaronetal} allows us to detect which values of $V_n(\mathcal{X}_n)$ are large enough to be incompatible with these two assumptions. Since a similar reasoning can be also applied if we consider the volume of $\hat{\delta}(C_{r}(\mathcal{X}_n)\setminus\mathcal{X}_n)$, the test is based on the opposite approach: Under ($f_{0,1}^L$) and ($R$), if the test statistic takes large enough values, it will mean that the selected $r$ is not appropriate and a smaller one should be considered.

The performance of this test can be illustrated using the real database of invasive plants in Azorean islands. Given the sample $\mathcal{X}_{740}$, the practitioner could be interested in testing the null hythothesis that the EOO is $r-$convex, for instance, for $r=5$. According to Figure \ref{fig2} (third row, right), it is clear that large Atlantic Ocean areas are inside $C_{5}(\mathcal{X}_{740})$ and the EOO is overestimated. Moreover, the volume of $\hat{\delta}(C_{5}(\mathcal{X}_{740})\setminus\mathcal{X}_{740})$ will be too large. In fact, although larger samples sizes were considered, its volume would take a constant value (see gray ball inside the EOO reconstruction). Therefore, the null hypothesis of $5-$convexity should be rejected. Note that the situation is the opposite if testing $r-$convexity for $r=0.3$ is the goal. The volume of $\hat{\delta}(C_{0.3}(\mathcal{X}_{740})\setminus\mathcal{X}_{740})$ should be clearly smaller. Furthermore, when the sample size increases, this volume tends to zero. Formally, the asymptotic behaviour of the test is stated in Theorem \ref{test}. 



\begin{theorem}\label{test}Let $r>0$ and let $\mathcal{X}_n$ be a random and i.i.d sample drawn according to a density $f$ that satisfies ($f_{0,1}^L$) with compact and nonempty support $S$ under ($R$). Let $\hat{f}_n$ be the corresponding density estimator introduced in Definition \ref{fn} and let $K$ be the kernel function under ($\mathcal{K}_{\phi}^p$). Assume that $h_n = O(n^{-\zeta})$ for some $0 <\zeta< 1/d$. For the following decision problem,
	$$H_0:\mbox{ }S\mbox{ is }r-\mbox{convex versus }H_1:\mbox{ }S\mbox{ is not }r-\mbox{convex}.$$
	\begin{itemize}
	\item[(a)] The test based on the statistic $\hat{V}_{n,r}=\hat{\delta}(C_{r}(\mathcal{X}_n)\setminus\mathcal{X}_n)^d$ with critical region $RC=\{\hat{V}_{n,r}> c_{n,\alpha}\}$, where
	$$c_{n,\alpha}=\frac{1}{n}(-log(-log(1-\alpha))+log(n)+(d-1)log(log(n))+log(\beta))$$ has an asymptotic level less than $\alpha$.\\
	\item[(b)] Moreover, if $S$ verifying ($R$) is not $r-$convex, the power is $1$ for sufficiently large $n$.
\end{itemize}
\end{theorem}
\begin{remark}
Note that the optimal kernel sequence size, $h_n=h_0n^{1/(d+4)}$, satisfies the hypotheses under which Theorem \ref{test} holds. Therefore, any reasonable bandwidth selector should be suitable for testing $r-$convexity.
\end{remark}

\subsection{Selection and consistency results of the optimal smoothing parameter}\label{resultadosmaximo2}

The optimal estimation of the smoothing parameter $r_0$ from  $\mathcal{X}_n$ is based on the test previously proposed. Specifically, according to Definition \ref{sup}, $r_0$ will be estimated by
\begin{equation}\label{r0estimador}
\hat{r}_{0}=\sup\{\gamma>0:  \mbox{ The null hypothesis }H_0 \mbox{ that } S \mbox{ is } \gamma-\mbox{convex is accepted}\}.
\end{equation}
That is, it is proposed to select the
largest value of $\gamma$ compatible with the $\gamma-$convexity assumption. Note that this choice depends on the significance level of the test. Again, we use the example of invasive plants in Azorean islands in order to analyze this estimator. Under ($f_{0,1}^L$) and ($R$), if the volume of $\hat{\delta}(C_{\gamma}(\mathcal{X}_n)\setminus \mathcal{X}_n)$ is large enough, then the null hypothesis of $\gamma-$convexity will be rejected. Therefore, a smaller value of $\gamma$ should be selected. This case corresponds to Figure \ref{fig2} (third row, right) taking $\gamma=5$. However, the situation is completely opposite in Figure \ref{fig2} (second row, right) when $\gamma=0.03$. Here, the size of the maximal spacing found in $C_{0.03}(\mathcal{X}_{740})\setminus \mathcal{X}_{740}$ does not allow to reject that the support is $0.03-$convex. As a consequence, a bigger $\gamma$ than $0.03$ should be considered.

The technical properties for the estimator of $r_0$ are considered next. First, the existence of the supreme defined in (\ref{r0estimador}) must be guaranteed, a result which is proved in Theorem \ref{consistencia}. In addition, it is also proved that $\hat{r}_0$ consistently estimates $r_0$.

\begin{theorem}\label{consistencia}Let $f$ be a density function that satisfies ($f_{0,1}^L$) with compact, nonconvex and nonempty support $S$ under ($R$). Let $\hat{f}_n$ be the density estimator introduced in Definition \ref{fn} and let $K$ be the kernel function under ($\mathcal{K}_{\phi}^p$). Assume that $h_n = O(n^{-\zeta})$ for some $0 <\zeta< 1/d$. Let $r_0$ be the parameter
	defined in (\ref{maximo2}) and $\hat{r}_{0}$ defined in (\ref{r0estimador}). Let $\{\alpha_n\}\subset (0,1)$ be a sequence converging to zero such that $\log(\alpha_n)/n\rightarrow 0$. Then, $\hat{r}_0$ converges to $r_0$ in probability.
\end{theorem}

\begin{remark}
For the sake of clarity, $S$ is assumed non-convex throughout the test. However, if $S$ is convex, it can be shown that
	$\hat{r}_0$ goes to infinity (which is the value of $r_0$ in this case)
	because, with high probability, the test is not rejected for all values of $r$.
\end{remark}

\section{Consistency of resulting support estimator}\label{cons}

The behaviour of the random set $C_{\hat{r}_0}(\mathcal{X}_n)$ as an estimator of $S$ can be studied once the consistency of $\hat{r}_0$ has been proved. Two metrics between sets are usually considered in order to assess the performance of a support estimator. Specifically, let $A$
and $C$ be two closed, bounded, nonempty subsets of $\mathbb{R}^{d}$. The
Hausdorff distance between $A$ and $C$ is defined by\vspace{-0.13cm}
$$
d_{H}(A,C)=\max\left\{\sup_{a\in A}d(a,C),\sup_{c\in C}d(c,A)\right\},\vspace{-0.13cm}
$$where $d(a,C)=\inf\{\|a-c\|:c\in C\}$ and $\|\mbox{ }\|$ denotes the Euclidean norm. Besides, if $A$ and $C$
are two bounded and Borel sets then the distance in measure between $A$ and $C$ is defined by $d_{\mu}(A,C)=\mu(A\triangle C)$, where $\mu$ denotes the Lebesgue measure and $\triangle$, the symmetric difference, that is, $A\triangle C=(A \setminus C)\cup(C \setminus A). $
Hausdorff distance quantifies the physical proximity between two sets whereas the distance in measure
is useful to quantify their similarity in content. However, neither of these distances are completely useful for measuring
the similarity between the shape of two sets. The Hausdorff distance between boundaries, $d_H(\partial A,\partial C)$, can be also used to evaluate the
performance of the estimators (see Ba\'illo and Cuevas, 2001; Cuevas and Rodr\'iguez-Casal, 2004; Rodr\'iguez-Casal, 2007 or Genovese et al., 2012).

In particular, if $\lim_{r\rightarrow r_0^+}d_H(S, C_{r}(S))=0$ then, the
consistency of $C_{\hat{r}_0}(\mathcal{X}_n)$ can be proved easily from Theorem \ref{consistencia}. However, the
consistency cannot be guaranteed if $d_H(S, C_{r}(S))$ does not go to zero as $r$ goes to $r_0$ from above (as $\hat{r}_0$ does, see Proposition \ref{alpha} below).
This problem can be solved by considering the estimator
$C_{r_n}(\mathcal{X}_n)$ where $r_n=\nu \hat{r}_0$ with $\nu\in (0,1)$ fixed. This ensures that, for $n$ large enough, with high probability,
$C_{r_n}(\mathcal{X}_n)\subset S$. From the practical point of view the selection of $\nu$ is not a major issue because
$\hat{r}_0$ is numerically approximated and the
computed estimator always satisfies this property without multiplying by $\nu$. In some sense, Theorem \ref{consistencia2}
gives the convergence rate of the numerical approximation of $\hat{r}_0$.

\begin{theorem}\label{consistencia2}Let $\mathcal{X}_n$ be a random and i.i.d sample drawn according to a density $f$ that satisfies ($f_{0,1}^L$) with compact, nonconvex and nonempty support $S$ under ($R$). Let $r_0$ be the parameter defined in (\ref{maximo2})  and $\hat{r}_{0}$ defined in (\ref{r0estimador}). Let $\{\alpha_n\}\subset (0,1)$ be a sequence converging to zero such that $\log(\alpha_n)/n\rightarrow 0$. Let be $\nu \in (0,1)$ and $r_n=\nu \hat{r}_0$. Then,  
	$$
	d_H(S, C_{r_n}(\mathcal{X}_n))=O_P\left(\frac{\log n}{n}\right)^{\frac{2}{d+1}}.
	$$
	The same convergence order holds for $d_H(\partial S, \partial C_{r_n}(\mathcal{X}_n))$ and $d_\mu(S\triangle C_{r_n}(\mathcal{X}_n))$.
\end{theorem}

 \section{Numerical illustration}\label{numerical}

 The main numerical aspects of the estimation algorithm of $r_0$ in (\ref{maximo2}) are detailed in what follows. Although  the method proposed in this work is fully data-driven from a theoretical point of view, its practical implementation depends on the specification of two parameters to be selected by the practitioner: the significance level of the test $\alpha$ and the maximum number for connected components $\mathcal{C}$ of the resulting support estimator. Choosing them is a much more flexible and simpler problem than the specification of the shape index $r_0$.
 
 With probability one, for a large enough $n$, the existence of the estimator $\hat{r}_0$ defined in (\ref{r0estimador}) is guaranteed under the hypotheses of Theorem \ref{consistencia}. However, in practice, this estimator might not exist for a specific sample $\mathcal{X}_n$ and a given value of
 the significance level $\alpha$. Therefore, the influence of $\alpha$ must be taken into account. The null hypothesis of $r-$convexity
 will be (incorrectly) rejected for $0< r\leq r_0$ with
 probability $\alpha$, approximately. This is not important from the theoretical point of view, since we are assuming that $\alpha=\alpha_n$ goes to zero as
 the sample size increases. But, what should be done, for a given sample, if $H_0$ is rejected for {\it all} $r$ (or at least {\it all} reasonable values of $r$)? In order to fix a minimum acceptable
 value of $r$, it is assumed that $S$ (and, hence, its estimator) will
 have no more than $\mathcal{C}$ connected components. Too fragmented estimators will not be considered even in the case that we reject $H_0$ for all $r$. The minimum value
 that ensures a number of connected components not greater than $\mathcal{C}$ will be taken in this latter case. Therefore, this parameter $\mathcal{C}$ can be interpreted as a geometric stopping criteria that does not appear in theoretical results because the sequence $\alpha_n$ is assumed to tend to zero.

 Dichotomy algorithms can be used to compute $\hat{r}_0$. The practitioner
 must select a maximum number of iterations $I$ and two initial points $r_m$ and $r_M$ with $r_m<r_M$ such that the null hypothesis of $r_M-$convexity is rejected and the null hypothesis of $r_m-$convexity is accepted. According to the
 previous comments, it is assumed that the number of connected components of $C_{r_m}(\mathcal{X}_n)$ must not be greater than $\mathcal{C}$. Choosing a value close
 enough to zero is usually sufficient to select $r_m$. According to Figure \ref{fig2} (second row, right), the maximal spacing in $C_{0.03}(\mathcal{X}_n)$ will be small enough to accept $0.03-$convexity. Therefore, taking $r_m\leq 0.03$ will be a good choice. However, if selecting this $r_m$ is not possible because, for very low values of $r$, the hypothesis
 of $r-$convexity is still rejected then $r_0$ is estimated as the positive closest value to zero $r$ such
 that the number of connected components of $C_r(\mathcal{X}_n)$ is smaller than or equal to $\mathcal{C}$. On the other hand, if a large enough spacing for having a statistically significant test cannot be found in $H(\mathcal{X}_{n})$ then we propose $H(\mathcal{X}_n)$ as the estimator for the support.
 
 To sum up, the following inputs should be given: the significance level $\alpha\in (0,1)$, a maximum number of iterations $I$, a maximum number of connected components $\mathcal{C}$ and two initial
 values $r_m$ and $r_M$. Given these parameters $\hat{r}_0$ will be computed as follows:  \vspace{.08cm}
 
 \begin{enumerate}
 	\item In each iteration and while the number of them is smaller than $I$:\vspace{.08cm}
 	\begin{enumerate}
 		\item $r=(r_m+r_M)/2.$\vspace{.08cm}
 		\item If the null hypothesis of $r-$convexity is not rejected then $r_m=r$.\vspace{.08cm}
 		\item Otherwise, $r_M=r$.\vspace{.08cm}
 	\end{enumerate}
 	\item Then, $\hat{r}_0=r_m$.\vspace{.08cm}
 \end{enumerate}

 Some technical aspects related to the computation of the
 maximal spacings must be also mentioned. In the proposed procedure, the null hypothesis needs to be tested $I$ times. Since it involves the  calculation of the maximal spacing, one may be aware of computational cost of the method. Nevertheless, as noted by Rodr\'iguez-Casal and Saavedra-Nieves (2016), this maximal spacing does not need to be specifically determined and it is enough to check if there exists a point $x$ such that $$x+\frac{c_{n,\alpha}^{1/d}}{\hat{f}_n^{1/d}(x)}A\subset C_r(\mathcal{X}_n)\setminus \mathcal{X}_n.$$In this case, $\hat{V}_{n,r}\geq c_{n,\alpha}$ and, therefore, the null hypothesis of $r-$convexity will be rejected. Furthermore, note that if this disc exists then $x\notin B_{c_{n,\alpha}^{x,w}}(X_k)$ where $c_{n,\alpha}^{x,w}=c_{n,\alpha}^{1/d}w_d^{-1/d}\hat{f}_n^{-1/d}(x)$ and $X_k$ denotes the sample point such that $x\in Vor(X_k)$. Therefore, $\hat{f}_n(x)=f_n(X_k)$.\\ 
 Then, the centers of the possible maximal balls that belong to the Voronoi tile with nucleus $X_i$ ($i,\cdots,n$) necessarily
 lie in $B_{c_{n,\alpha}^{X_i,w}}(X_i)^c\cap Vor(X_i)$. 
 We will follow the next steps:\vspace{.08cm}
 \begin{enumerate}
 	\item Determine the set of candidates for ball centers  $D(r) = C_{r}(\mathcal{X}_n)\cap \bigcup_{X_i\in E(m)} (\partial B_{c_{n,\alpha}^{X_i,w}}(X_i)\cap Vor(X_i))$ where $E(m)\subset \mathcal{X}_n$ denotes the extremes of the $m-$shape of $\mathcal{X}_n$ when $m=\min\left\{c_{n,\alpha}^{X_j,w}:X_j\in \mathcal{X}_n\right\}$, see Edelsbrunner (2014). If $x\in D(r)$ then we can guarantee that $B_{c_{n,\alpha}^{X_i,w}}(x)\cap \mathcal{X}_n=\emptyset$. Equivalently,
 	$$x+\frac{c_{n,\alpha}^{1/d}}{\hat{f}_n^{1/d}(x)}A\subset C_r(\mathcal{X}_n\setminus \mathcal{X}_n ).$$\vspace{.08cm}
 	\item Calculate $M(r)=\max\{d(x,\partial C_{r}(\mathcal{X}_n):x\in D(r)\}$.\vspace{.08cm}
 	\item If $M(r)\leq \hat{c}_{n,\alpha}$ then the null hypothesis of $r-$convexity is not rejected.\vspace{.08cm}
 \end{enumerate}
 
 It should be noted that if $X_i\notin E(m)$, for all $x\in B_{c_{n,\alpha}^{X_i,w}}(X_i)^c\cap Vor(X_i)$, $B_{c_{n,\alpha}^{X_i,w}}(x)\cap \mathcal{X}_n\neq\emptyset$. Therefore, these points can be discarded in order to determine $D(r)$. Furthermore, $E(m)$, $\partial C_{r}(\mathcal{X}_n)$ and $\partial B_{\hat{c}_{n,\alpha,r}^*}(\mathcal{X}_n)$ can be easily computed (at least for the bidimensional case). See Pateiro-L\'opez and Rodr\'iguez-Casal (2010) for further details.

 \section{Extent of occurrence estimation}\label{reald}

The new support estimator  introduced in this work  will be used for reconstructing the EOO of an assemblage of terrestrial invasive plants in two islands of the Azores Archipelago, Terceira and S\~ao Miguel. For this real dataset, we have shown that convexity assumption is very restrictive. According to Figure \ref{fig2} (first and second rows, left), sea areas are inside the classical estimator of the EOO. Obviously, it is overestimated given that terrestrial invasive plants does not occupy the Atlantic Ocean. The goal here is to reconstruct the EOO overcoming these limitations.

First, it is necessary to estimate the optimal value $r_0$ from the sample of $740$ geographical locations. If we select the significance level $\alpha$ equal to $0.01$ and $\mathcal{C}= 4$, the resulting estimator is  $\hat{r}_0=0.127$. In Figure \ref{fig3}, $C_{\hat{r}_0}(\mathcal{X}_{740})$ is shown. According to the results obtained, the EOO reconstruction has two different connected components corresponding to the two Azorean islands. Unlike classical EOO estimator, sea areas are not inside the reconstruction. Therefore, if the sample size is large enough, a more sophisticated and realistic estimator of the EOO can be determined.

\begin{figure}\vspace{2cm}
	\hspace{-1.5cm} \includegraphics[scale=1]{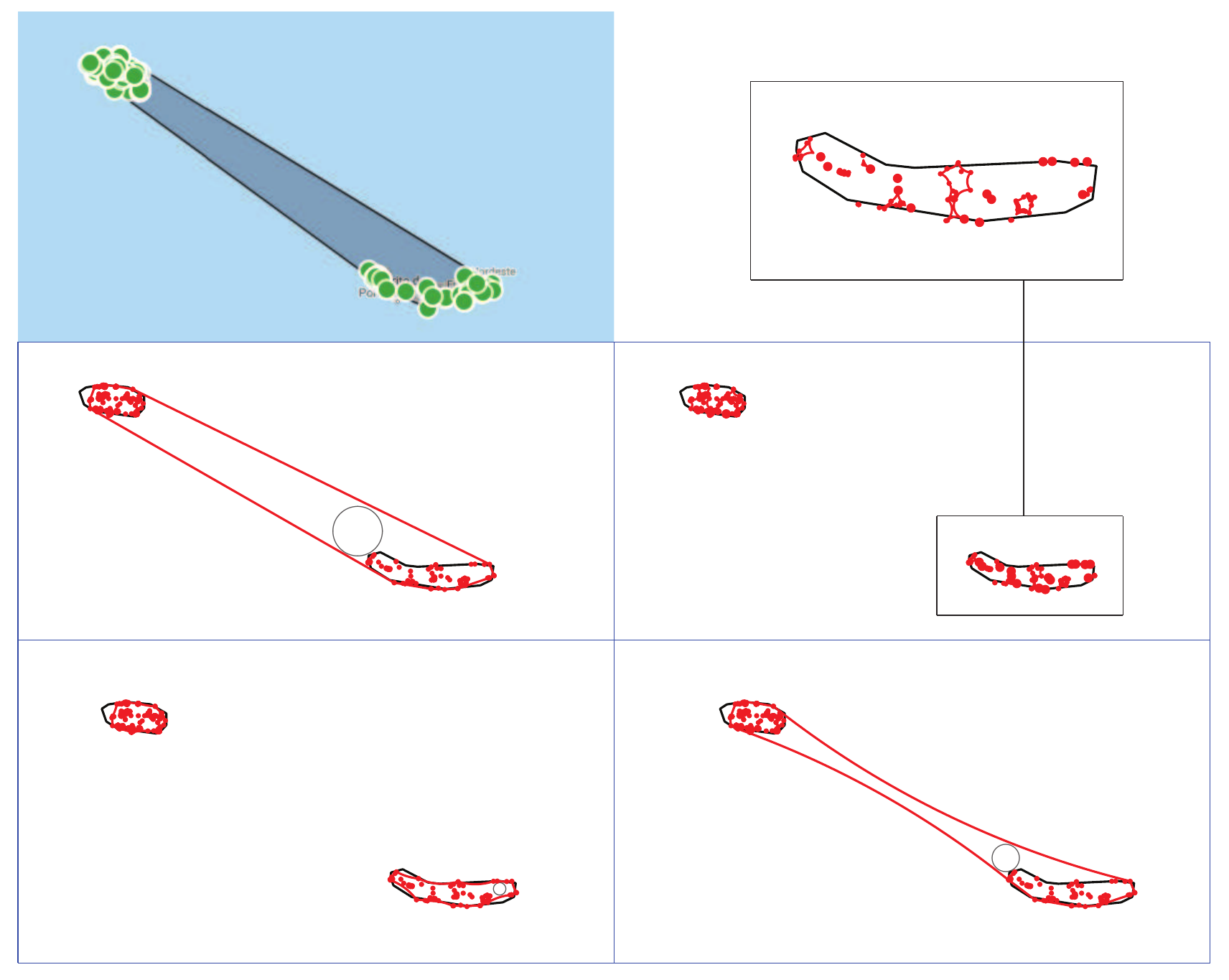}  
	\caption{In the first row, GeoCAT reconstruction of the EOO determined from the sample of $740$ geographical locations in two Azorean islands (left). In the second row (red color), $H(\mathcal{X}_{740})$ (left) and $C_{0.03}(\mathcal{X}_{740})$ (right). In the third row (red color), $C_{0.3}(\mathcal{X}_{740})$ (left) and 
		$C_{5}(\mathcal{X}_{740})$ (right).}\label{fig2}
\end{figure}

\begin{figure} 
	\includegraphics[scale=1]{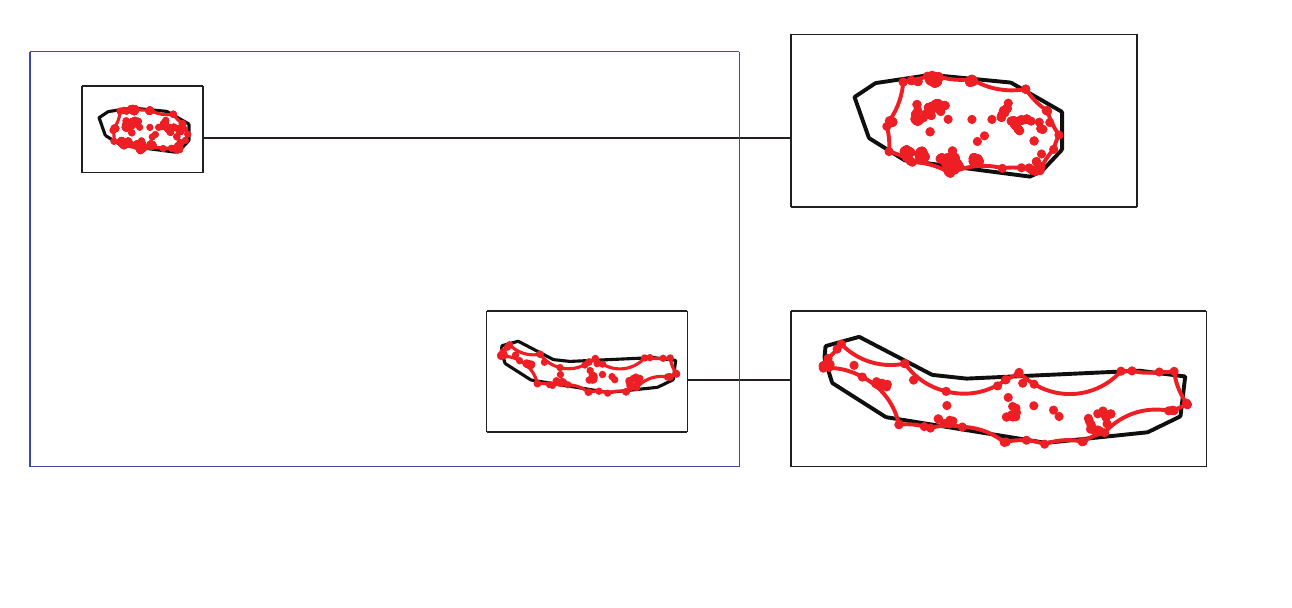}  
	\caption{EOO estimator, $C_{\hat{r}_0}(\mathcal{X}_{740})$ where $\hat{r}_0=0.127$.}\label{fig3}
\end{figure}

 \begin{figure}
\hspace{-.5cm}  	\includegraphics[scale=1]{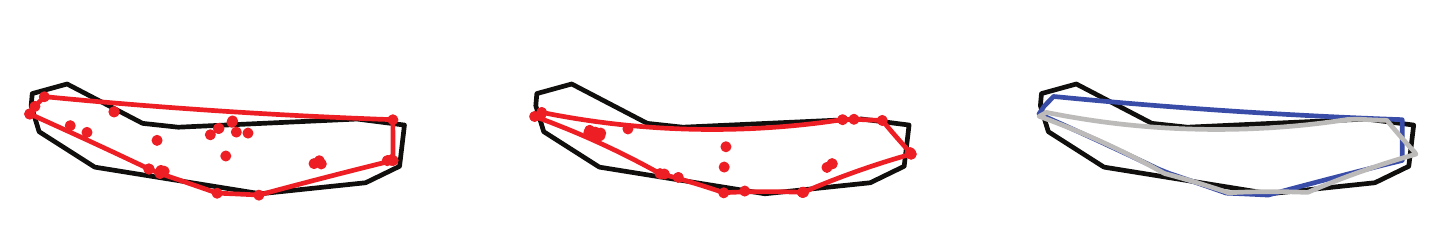}  
 	\caption{EOO estimator in 2015, $H(\mathcal{X}_{33})$ (left); EOO estimator in 2016, $C_{\hat{r}_0}(\mathcal{X}_{48})$ where $\hat{r}_0=1.5$ (center); EOO estimators in 2015 (blue) and 2016 (gray) (right).}\label{fig4}
 \end{figure}

The new method, although designed for handling more complex situations, provides similar reconstructions to those corresponding to the convex hull in those cases where the classical reconstruction works appropiately. For showing this, we will focus on the geographical locations from S\~ao Miguel island. Separately, the EOO will be estimated from data corresponding to years $2015$ and $2016$. A total of $33$ and $48$ geographical locations are available in $2015$ and $2016$, respectively. 

Figure \ref{fig4} contains the EOO estimator in $2015$ (left) and $2016$ (center). In 2015, the resulting reconstruction of the EOO is equal to $H(\mathcal{X}_{33})$. In 2016, $\hat{r}_0=1.5$; however, the estimation of the EOO obtained, $C_{1.5}(\mathcal{X}_{48})$, is not so different from the convex hull. This last illustration suggests that, if more amount of data are available by year, this kind of analysis could be useful for studying the temporal changes in the spatial pattern of organisms, including invasive plants, on an area of interest.
 
\section{Conclusions and open problems}\label{con}

The main goal of this work is to propose a new data-driven method for reconstructing a $r-$convex support in a consistent way. The route designed to reach this goal can be summarized as follows: (1) Defining the optimal value of $r$, $r_0$, to be estimated, (2) establishing a nonparametric test to asess the null hypothesis that $S$ is $r-$convex for a given $r>0$, (3) defining the estimator of $r_0$ that strongly relies on the previous test and (4) checking that the estimator of $r_0$ and the resulting support reconstruction are consistent.

The definition of the estimator $\hat{r}_0$ depends on the $r-$convexity test established that, of course, could be used in an independent way. In many practical situations where the support is completely unknown and only a sample of points is available, it can be interesting to test if the corresponding support distribution is $r-$convex.

Futhermore, the behaviour of the proposed support estimator was illustrated through the estimation of the EOO of an assemblage of terrestrial invasive plants in two Azorean islands. In this particular case, where convexity assumption on the EOO is too restrictive, our support estimator provides a more realistic and sophisticated reconstruction. Besides, we have shown that when the classical convex reconstruction works appropiately, our estimator offers similar reconstructions. Furthermore, we have shown that estimating the EOO from annual (or any other time period) occurrences could be useful for detecting temporal changes in the spatial pattern of organisms.

Note that the resulting support estimator is spatially
flexible. In other words, it is able to distinguish the different disconnected components of the support. Therefore, it could be used for reconstructing the support of an intensity function of a Poisson process.

Finally, another interesting problem and intimately related to the EOO reconstruction is to estimate the area of occupancy (AOO). The IUCN defined the AOO as the area within its extent of occurrence. Under $r-$convexity, we could estimate the AOO as the area of the $r-$convex hull of the sample points. However, this estimator suffers from the drawback of not being rate-optimal. Arias-Castro et al. (2018) proposed an optimal volume estimator based on the sample $r-$convex hull using a sample splitting strategy that attains the minimax lower bound. Therefore, the problem of estimating the AOO could be studied from a different perspective in future.

 \section{Proofs}\label{proof}
 In this section the proofs of the stated theorems are presented. 
 
$ $\\\emph{Proof of Theorem \ref{aaronetal}.}\vspace{2mm}\\
 	
 	First, Aaron et al. (2017) assumed that $f$ is H\"older continuous with respect to Lebesgue measure. Under ($f_{0,1}^L$), this condition is satisfied. See Aaron et al. (2017) for more details.

 	Furthermore, Aaron et al. (2017) also assumed that there exists $k<d$ and $C_{\partial S}>0$ such that $N(\partial S,\epsilon)\leq C_{\partial S}\epsilon^{-k}$ where $N(\partial S,\epsilon)$ denotes the inner covering number of $\partial S$. Under ($R$), Theorem 1 in Walther (1997) guaranteed that $\partial S$ is a $\mathcal{C}^1$ $(d-1)-$dimensional submanifold. Therefore, the previous assumption is fulfilled for $k = d-1$. See Aaron et al. (2017) for more details. 	
 	 
$ $\\\emph{Proof of Theorem \ref{test}.}\vspace{2mm}\\
First, we will prove (a) and then, (b).
\begin{itemize}
		\item[(a)] Under $H_0$ ($C_{r}(S)=S$), $C_{r}(\mathcal{X}_n)\subset S$. Then,
		$$\hat{\delta}(C_r(\mathcal{X}_n)\setminus\mathcal{X}_n)\leq\sup\left\{\gamma:\exists x\mbox{ such that }\{x\}+\frac{\gamma }{\hat{f}_n(x)^{1/d}}A\subset S\setminus \mathcal{X}_n\right\}.$$
	If we apply Lemma \ref{Lemma5adaptado}, we get, with probability one, for $n$ large enough,
	$$\hat{\delta}(C_{r}(\mathcal{X}_n)\setminus\mathcal{X}_n)\leq\sup\left\{\gamma:\exists x\mbox{ such that }\{x\}+\frac{(1-\epsilon_n^+)\gamma }{f(x)^{1/d}}A\subset S\setminus \mathcal{X}_n\right\}.$$
	Equivalently, $\Delta_n(\mathcal{X}_n)\geq(1-\epsilon_n^+)\hat{\delta}(C_{r}(\mathcal{X}_n)\setminus\mathcal{X}_n) $ and, therefore, $\mathbb{P}(\hat{V}_{n,r} > c_{n,\alpha})\leq \mathbb{P}(V_n(\mathcal{X}_n)> (1-\epsilon_n^+)c_{n,\alpha}) $ from where it follows that $\mathbb{P}(\hat{V}_{n,r}> c_{n,\alpha})$ can be majorized by,
	$$\mathbb{P}(U(\mathcal{X}_n)>-(1-\epsilon_n^+)^d\log(-\log (1-\alpha)  ) $$
	$$+ ((1-\epsilon_n^+)^d-1)(\log(n)
	+(d-1)\log(\log(n)) +\log(\beta)  )   ).$$According to Theorem \ref{aaronetal}, $U(\mathcal{X}_n)\stackrel{d}{\rightarrow} U$ when $n\rightarrow\infty$. Furthermore, notice that $U$ has a continuous distribution, so convergence in distribution implies that
	$$
	\sup_{u}\left|\mathbb{P}\left(U(\mathcal{X}_n)\leq u\right)-\mathbb{P}(U\leq u)\right|\to 0.
	$$Therefore, using that $\log(n)\epsilon_n^+$ tends to zero, we get that
	$$ \mathbb{P} (U>-\log(-\log (1-\alpha)  )+o(1))\rightarrow \alpha.$$As a consequence,
	$$\mathbb{P}(\hat{V}_{n,r} >c_{n,\alpha})\leq \mathbb{P} (U(\mathcal{X}_n)>-\log(-\log (1-\alpha)  )+o(1))\rightarrow \alpha.$$
	\item[(b)] From Lemma \ref{Lemma5adaptado} (ii), 
	$$\hat{\delta}(C_{r}(\mathcal{X}_n)\setminus\mathcal{X}_n)\geq (\lambda_0-\epsilon_n^-)\mathcal{R}(C_{r}(\mathcal{X}_n)\setminus\mathcal{X}_n),$$where $\epsilon_n^-$ tends to zero, almost surely. Under $H_1$ ($S$ is not $r-$convex, $S\subsetneq  C_{r}(S)$), we will prove that, with probability one and for $n$ large enough,
	$$\mathcal{R}(C_{r}(\mathcal{X}_n)\setminus\mathcal{X}_n)\geq \rho^{'}>0. $$In particular, we will find a closed ball of radius $\rho^{'}>0$ that, with probability one and for $n$ large enough, is inside $C_{r}(\mathcal{X}_n)\setminus\mathcal{X}_n$.\\
	Then, let be $r^{*}$ such that $r>r^*>0$ and $S\subsetneq C_{r^*}(S)\subset C_{r}(S)$. Since $S$ is under ($R$), Proposition 2.2 in Rodr\'iguez-Casal and Saavedra-Nieves (2016) ensures that $S$ is $r-$convex. However, under $H_1$, $S$ is not $r-$convex. Therefore, it is easy to guarantee the existence of $r^{*}$.\\
	According to Lemma 8.3 in Rodr\'iguez-Casal and Saavedra-Nieves (2016),
	$$\exists B_\rho(x) \mbox{ such that }B_\rho(x)\subset C_{r^*}(S)\mbox{ and }B_\rho(x)\cap S=\emptyset. $$
	It can be assumed, without loss
	of generality, that $r\leq\frac{\rho}{2}+r^*$. If this is not the case then it would be possible to replace $r^*$ by $r^{**}>r^*$ satisfying  $r^{**}<r\leq\frac{\rho}{2}+r^{**}$. For this $r^{**}$,
	$$B_\rho(x)\subset C_{r^*}(S)\subset  C_{r^{**}}(S) \mbox{ and }B_\rho(x)\cap S=\emptyset. $$
	Now, we can apply Lemma 3 in Walther (1997) in order to ensure that 
	$$
	\mathbb{P}\left(S\oplus r^* B_1[0]\subset \mathcal{X}_n\oplus rB_1[0],\mbox{ eventually}\right)=1.\vspace{.3cm}
	$$
	If $S\oplus r^* B_1[0]\subset \mathcal{X}_n\oplus rB_1[0]$ then $(S\oplus r^* B_1[0])\ominus r^*B_1[0]\subset (\mathcal{X}_n\oplus rB_1[0])\ominus r^*B_1[0]$, that is, $C_{r^*}(S)\subset  (\mathcal{X}_n\oplus rB_1[0])\ominus r^*B_1[0]$.
	This imply that
	$$C_{r^*}(S)\ominus (r-r^*)B_1[0]\subset ((\mathcal{X}_n\oplus rB_1[0])\ominus r^*B_1[0])\ominus (r-r^*)B_1[0].$$In addition,
	$$ ((\mathcal{X}_n\oplus rB_1[0])\ominus r^*B_1[0])\ominus (r-r^*)B_1[0]=(\mathcal{X}_n\oplus rB_1[0])\ominus rB_1[0]=C_{r}(\mathcal{X}_n),$$
	where we have used that, for sets $A,C$ and $D$, $(A\ominus C)\ominus D=A\ominus (C\oplus D)$. Finally, since $B_{\rho}(x)\subset C_{r^*}(S)$ and
	$\rho/2\geq (r-r^*)$, we have
	$B_{\rho/3}[x]\subset C_{r^*}(S)\ominus (\rho/2 )B_1[0]\subset C_{r^*}(S)\ominus (r-r^*)B_1[0]\subset C_{r}(\mathcal{X}_n)$. This part of the proof is concluded by taking $\rho^{'}=\rho/3$. \\
	Therefore, 
	$$\hat{\delta}(C_{r}(\mathcal{X}_n)\setminus\mathcal{X}_n)\geq (\lambda_0-\epsilon_n^-)\mathcal{R}(C_{r}(\mathcal{X}_n)\setminus\mathcal{X}_n)\geq(\lambda_0-\epsilon_n^-)\rho^{'}.$$
	Then, with probability one and for $n$ large enough,
	$$\hat{\delta}(C_{r}(\mathcal{X}_n)\setminus\mathcal{X}_n)>\frac{\rho^{'}}{2}\lambda_0.$$The proof is finished taking into account that $c_{n,\alpha}$ tends to zero.\qedhere
\end{itemize}

$ $\\
\emph{Proof of Theorem \ref{consistencia}.}\vspace{2mm}\\
Some auxiliary results are necessary. First we will prove that, with probability tending to one, $\hat{r}_0$ is at least as big as $r_0$.
\begin{proposition}\label{alpha}
	Let $f$ be a density function that fulfils condition ($f_{0,1}^L$) with compact, nonconvex and nonempty support $S$ under ($R$). Let $\hat{f}_n$ be the corresponding density estimator introduced in Definition \ref{fn} and let $K$ be the kernel function under ($\mathcal{K}_{\phi}^p$). Assume that $h_n = O(n^{-\zeta})$ for some $0 <\zeta< 1/d$. Let $r_0$ be the parameter
	defined in (\ref{maximo2}) and $\hat{r}_{0}$ defined in (\ref{r0estimador}). Let $\{\alpha_n\}\subset (0,1)$ be a sequence converging to zero. Then,
	$$\lim_{n\to\infty}\mathbb{P}(\hat{r}_0\geq r_0)=1.$$\end{proposition}

\begin{proof}Equivalently, we will prove that 
	$$\lim_{n\to\infty}\mathbb{P}(\hat{r}_0<r_0)=0.$$
	From the definition of $\hat{r}_0$, see (\ref{r0estimador}), it is clear that
	$$\mathbb{P}(\hat{r}_0\geq r_0)\geq \mathbb{P}(\hat{V}_{n,r_0}\leq c_{n,\alpha_n})$$where $\hat{V}_{n,r_0}=\hat{\delta}(C_{r_0}(\mathcal{X}_n)\setminus \mathcal{X}_n )^d$ and
	$c_{n,\alpha_n}= n^{-1}( -\log(-\log(1-\alpha_n) ) +\log (n)   +(d-1)\log\log(n)+\log{\beta})$. Therefore,
	$$\mathbb{P}(\hat{r}_0<r_0)\leq \mathbb{P}(\hat{V}_{n,r_0}> c_{n,\alpha_n}).$$
	Since, with probability one, $C_{r_0}(\mathcal{X}_n)\subset S$, applying Lemma \ref{Lemma5adaptado}, $\Delta_n(\mathcal{X}_n)\geq(1-\epsilon_n^+)\hat{\delta}(C_{r_0}(\mathcal{X}_n)\setminus\mathcal{X}_n) $ and, therefore, $\mathbb{P}(\hat{V}_{n,r_0} > c_{n,\alpha_n})\leq \mathbb{P}(V_n(\mathcal{X}_n)> (1-\epsilon_n^+)^d c_{n,\alpha_n}) $ from where it follows that $\mathbb{P}(\hat{V}_{n,r_0}> c_{n,\alpha_n})$ can be majorized by,
	\small{$$\mathbb{P}(U(\mathcal{X}_n)>-(1-\epsilon_n^+)^d\log(-\log (1-\alpha_n)  ) + ((1-\epsilon_n^+)^d-1)(\log(n)+(d-1)\log(\log(n)) +\log(\beta)  )   ).$$}
	According to Theorem \ref{aaronetal}, $U(\mathcal{X}_n)\stackrel{d}{\rightarrow} U$ when $n\rightarrow\infty$. Furthermore, notice that $U$ has a continuous distribution, so convergence in distribution implies that
	$$
	\sup_{u}\left|\mathbb{P}\left(U(\mathcal{X}_n)\leq u\right)-\mathbb{P}(U\leq u)\right|\to 0.
	$$
	Since $\alpha_n\to 0$ and $\log(n)\epsilon_n^+ \to 0$, we can prove $$\mathbb{P}(U> -(1-\epsilon_n^+)^d\log(-\log (1-\alpha)  ) + ((1-\epsilon_n^+)^d-1)(\log(n)+(d-1)\log(\log(n)) +\log(\beta)  )   ) \to 0.$$This ensures
	that $$\mathbb{P}(U(\mathcal{X}_n)>-(1-\epsilon_n^+)^d\log(-\log (1-\alpha)  ) + ((1-\epsilon_n^+)^d-1)(\log(n)+(d-1)\log(\log(n)) +\log(\beta)  )   ) \to 0.$$Therefore, $\mathbb{P}(\hat{r}_0\geq r_0)\to 1$.

\end{proof}

It remains to prove that $\hat{r}_0$ cannot be arbitrarily larger that $r_0$. Some auxiliary results must be proved next.

\begin{lemma}\label{lem_spacing_minimo}Let $\mathcal{X}_n$ be a random and i.i.d sample drawn according to a density $f$ that satisfies ($f_{0,1}^L$) with compact, nonconvex and nonempty support $S$ under ($R$). Let $r_0$ be the parameter defined in (\ref{maximo2}). Then, for all $r>r_0$, there exists an open ball  $B_\rho(x)$
	such that
	$B_{\rho}(x)\cap S=\emptyset$ and
	$$
	\mathbb{P}\left(B_{\rho}(x)\subset C_{r}(\mathcal{X}_n),\mbox{ eventually}\right)=1.
	$$
	
\end{lemma}

\begin{proof}\normalsize
	Let be $r^{*}$ such that $r>r^*>r_0$. Since $C_{r_0}(S)=S \subsetneq C_{r^*}(S)$, according to Lemma 8.3 in Rodr\'iguez-Casal and Saavedra-Nieves (2016),
	$$\exists B_\epsilon(x) \mbox{ such that }B_\epsilon(x)\subset C_{r^*}(S)\mbox{ and }B_\epsilon(x)\cap S=\emptyset. $$
	It can be assumed, without loss
	of generality, that $r\leq\frac{\epsilon}{2}+r^*$. If this is not the case then it would be possible to replace $r^*$ by $r^{**}>r^*$ satisfying  $r^{**}<r\leq\frac{\epsilon}{2}+r^{**}$. For this $r^{**}$,
	$$B_\epsilon(x)\subset C_{r^*}(S)\subset  C_{r^{**}}(S) \mbox{ and }B_\epsilon(x)\cap S=\emptyset. $$
	Now, we can apply Lemma 3 in Walther (1997) in order to ensure that\\
	$$
	\mathbb{P}\left(S\oplus r^* B\subset \mathcal{X}_n\oplus rB,\mbox{ eventually}\right)=1.\vspace{.3cm}
	$$
	If $S\oplus r^* B\subset \mathcal{X}_n\oplus rB$ then $(S\oplus r^* B)\ominus r^*B\subset (\mathcal{X}_n\oplus rB)\ominus r^*B$, that is, $C_{r^*}(S)\subset  (\mathcal{X}_n\oplus rB)\ominus r^*B$.
	This imply that
	$$C_{r^*}(S)\ominus (r-r^*)B\subset ((\mathcal{X}_n\oplus rB)\ominus r^*B)\ominus (r-r^*)B.$$In addition,
	$$ ((\mathcal{X}_n\oplus rB)\ominus r^*B)\ominus (r-r^*)B=(\mathcal{X}_n\oplus rB)\ominus rB=C_r(\mathcal{X}_n),$$
	where we have used that, for sets $A,C$ and $D$, $(A\ominus C)\ominus D=A\ominus (C\oplus D)$. Finally, since $B_{\epsilon}(x)\subset C_{r^*}(S)$ and
	$\epsilon/2\geq (r-r^*)$, we have
	$B_{\epsilon/2}(x)\subset C_{r^*}(S)\ominus (\epsilon/2 )B\subset C_{r^*}(S)\ominus (r-r^*)B\subset C_{r}(\mathcal{X}_n)$. This concludes the proof of the lemma by taking $\rho=\epsilon/2$.  \end{proof}

\begin{proposition}\label{prop_nosobreestimamos}
	Let $\mathcal{X}_n$ be a random and i.i.d sample drawn according to a density $f$ that satisfies ($f_{0,1}^L$) with compact, nonconvex and nonempty support $S$ under ($R$). Let $r_0$ be the parameter defined in (\ref{maximo2}) and $\{\alpha_n\}\subset (0,1)$ a sequence converging to zero such that $\log(\alpha_n)/n\rightarrow 0$. Then, for any $\epsilon>0$,
	$$
	\mathbb{P}\left(\hat{r}_0\leq r_0+\epsilon,\mbox{ eventually}\right)=1.
	$$
\end{proposition}

\begin{proof}\normalsize
	Given $\epsilon>0$ let be $r=r_0+\epsilon$. According to Lemma \ref{lem_spacing_minimo}, there exists $x\in\mathbb{R}^d$ and $\rho>0$ such that $B_{\rho}(x)\cap S=\emptyset$
	and
	$$
	\mathbb{P}\left(B_{\rho}(x)\subset C_{r}(\mathcal{X}_n),\mbox{ eventually}\right)=1.
	$$
	Since, with probability one, $\mathcal{X}_n\subset S$ we have $B_{\rho}(x)\cap \mathcal{X}_n=\emptyset$. Then, $\{x\}+\rho B_{1}[0]\subset C_{r}(\mathcal{X}_n)\setminus \mathcal{X}_n$. Let $W$ the positive number $\hat{f}_n^{1/d}(x) w_d^{1/d}$. If $\gamma=\rho W>0$ then it is trivial to check that $$\{x\}+\frac{\gamma}{\hat{f}_n^{1/d}(x)}w_d^{-1/d}B_1[0]\subset C_{r}(\mathcal{X}_n)\setminus \mathcal{X}_n.$$
Therefore, $\hat{\delta}(C_{r}(\mathcal{X}_n)\setminus \mathcal{X}_n)\geq \gamma>0$ and, consequently, $\hat{V}_{n,r}=c_\gamma>0$.  Similarly, $\hat{V}_{n,r^{'}}\geq\hat{V}_{n,r}=c_\gamma>0$ for all $r^{'}\geq r$. On the other hand, since $-u_{\alpha_n}/\log(\alpha_n)=\log(-\log(1-\alpha_n))/\log(\alpha_n)\to 1$, we have, with probability one, 
$$\sup_{r^{'}}c_{n,\alpha_n}=c_{n,\alpha_n}\to 0.$$
Then, with probability one, there exists $n_0$ such that if $n\geq n_0$ we have $$\sup_{r^{'}}c_{n,\alpha_n}=c_{n,\alpha_n}=\frac{1}{n}(-log(-log(1-\alpha))+log(n)+(d-1)log(log(n))+log(\beta))<c_\gamma.$$ Therefore, $\hat{r}_0\leq r$. This last statement follows from
$\hat{V}_{n,r^{\prime}}>c_{n,\alpha_n}$ for all $r^{\prime}\geq r$ and the definition of $\hat{r}_0$, see (\ref{r0estimador}).
\vspace{.5cm}\end{proof}

Theorem \ref{consistencia} is a straightforward consequence of Propositions \ref{alpha} and \ref{prop_nosobreestimamos}.\qedhere\\

$ $\\
\emph{Proof of Theorem \ref{consistencia2}.}\vspace{2mm}\\Theorem 3 of Rodr\'{\i}guez-Casal (2007) ensures that, under ($R$) when $r=\lambda=\widetilde{r}$), then $\mathbb{P}(\mathcal{E}_n)\to 1$, where
	$$
	\mathcal{E}_n=\left\{d_H(S,C_{\widetilde{r}}(\mathcal{X}_n))\leq D\left(\frac{\log n}{n}\right)^{2/(d+1)}\right\},
	$$
	and $D$ is some constant. Under the hypothesis of Theorem \ref{consistencia2} this holds for any $\widetilde{r}\leq \min\{r,\lambda\}$. Fix one $\widetilde{r}\leq \min\{r,\lambda\}$ such
	that $\widetilde{r}<\nu r_0$ and define
	$\mathcal{R}_n=\{\widetilde{r}\leq r_n\leq r_0\}$. Since, by
	Theorem \ref{consistencia}, $r_n=\nu \hat{r}_0$ converges in probability to $\nu r_0$ and $\widetilde{r}<\nu r_0<r_0$, we have that $\mathbb{P}(\mathcal{R}_n)\to 1$.  If the events
	$\mathcal{E}_n$ and $\mathcal{R}_n$  hold (notice that $\mathbb{P}(\mathcal{E}_n\cap \mathcal{R}_n)\to 1$) we have $C_{\widetilde{r}}(\mathcal{X}_n)\subset C_{r_n}(\mathcal{X}_n)\subset S$ and, therefore,
	$$
	d_H(S,C_{r_n}(\mathcal{X}_n))\leq d_H(S,C_{\widetilde{r}}(\mathcal{X}_n))\leq D\left(\frac{\log n}{n}\right)^{2/(d+1)}.
	$$
	This completes the proof of the first statement of Theorem  \ref{consistencia2}. Similarly, it is possible to prove the result for the other error criteria considered in
	Theorem \ref{consistencia2}.

\section{Auxiliary results}\label{aux}

Lemma \ref{Lemma5adaptado} shows that Lemma 5 in Aaron et al. (2017) remains true if $S$ satisfies ($R$) and the density estimator $\hat{f}_n$ introduced in Definition \ref{fn} is considered. Concretely, Aaron et al. (2017) assumed that $S$ is a compact standard set. Roughly speaking, this condition prevents the support $S$ from being too spiky. Under the smoothness condition ($R$), standardness is guaranteed. See Rodr\'iguez-Casal (2007) or Cuevas and Fraiman (1997) for more details.

\begin{lemma}\label{Lemma5adaptado}Let $r>0$ and let $f$ be a density function  that satisfies ($f_{0,1}^L$) with compact and nonempty support $S$ under ($R$). Let $\hat{f}_n$ be the corresponding density estimator introduced in Definition \ref{fn} and let $K$ be the kernel function under ($\mathcal{K}_{\phi}^p$). Assume that $h_n=O(n^{-\zeta})$ with $\zeta\in (0,1/d)$. Then,
	\begin{itemize}
		\item [(i)] there exists a sequence $\epsilon_n^+$ such that $\log(n)\epsilon_n^+$ tends to zero and for all $x\in S$,
		$$\left(\frac{f(x)}{\hat{f}_n(x)}\right)^{1/d}\geq 1-\epsilon_n^+ \mbox{ }e.a.s.$$
		\item [(ii)]there exists a sequence $\epsilon_n^-$ tending to zero and a constant $\lambda_0$ such that for all $x\in C_{r}(\mathcal{X}_n)$, $(\hat{f}_n(x))^{1/d}\geq \lambda_0-\epsilon_n^-$ e.a.s.
		
	\end{itemize}
	
\end{lemma}

\begin{proof}Next, some preliminary results established in Aaron et al. (2017) (see proof of Lemma 5) are detailed.\\
	First, taking $\rho_n=\left(\frac{4f_1\log(n)}{f_0w_dn}\right)^{1/d}$, it can be proved that
	\begin{equation}\label{38}
	\mathbb{P}(d_H(\mathcal{X}_n,S)\geq \rho_n)\leq C_Sn^{-2},\mbox{ for }n\mbox{ large enough.}
	\end{equation} 
	Under ($\mathcal{K}_{\phi}^p$), $S$ verifies ($R$) and $K$ is bounded from below on a neighbourhood of the origin, there exist $c^{''}_K$ and $r_K>0$
	such that 
	\begin{equation}\label{39}
	\int_{S}K\left(\frac{u-x}{r}\right)du\geq c^{''}_Kr^d \mbox{ for all }x\in S \mbox{ and }r\leq r^{'}_K.
	\end{equation}
	Furthermore, for all $x\in S$,
	$$\mathbb{E}(f_n(x))=\int_{\{u:x+uh_n\in S\}}K(u)f(x+uh_n)du.$$
Since $f$ is Lipschitz and $\int_{\mathbb{R}^d}K(u)du=1$ it is verified that, for all $x\in S$,
\begin{equation}\label{40}
\mathbb{E}f_n(x)\leq \int_{\{u:x+uh_n\in S\}}K(u)(f(x)+k_f\|u\|h_n)du\leq f(x)+k_fh_nc_k
\end{equation}where $c_k>0$ y $k_f$ is established in Condition B.\\
From (\ref{39}) and the condition $f(x)>f_0$ for all $x\in S$, it follows that
\begin{equation}\label{41}
\mathbb{E}f_n(x)\geq f_0 c^{'}_K \mbox{ for all }x\in S.
\end{equation}

First, we will prove (i). Using triangular inequality, we can ensure that
\begin{equation}\label{42}
\max_{x\in S}(\hat{f}_n(x)-f(x))\leq \sup_{x\in S}|\hat{f}_n(x)-\mathbb{E}\hat{f}_n(x)|+\sup_{x\in S}(\mathbb{E}\hat{f}_n(x)-f(x)).
\end{equation}
As for the first term on the right hand side of this inequality, it is necessary to take into account that $K$ verifies ($\mathcal{K}_{\phi}^p$) and $h_n=O(n^{-\zeta})$ with $\zeta\in (0,1/d)$. Then, Theorem 2.3 in Gin\'e and Gillou (2002) guarantees that, there exists a constant $C_1$ such that, with probability one, for $n$ large enough, 
$$\sqrt{\frac{nh_n^d}{-log(h_n)}}\sup_{x\in\mathbb{R}^d}|f_n(x)-\mathbb{E}f_n(x)|\leq C_1.$$Therefore,
$$\sqrt{\frac{nh_n^d}{-log(h_n)}}\sup_{x\in\mathcal{X}_n}|f_n(x)-\mathbb{E}f_n(x)|\leq C_1.$$As a consequence,
\begin{equation}\label{43}
\sqrt{\frac{nh_n^d}{-log(h_n)}}\sup_{x\in S}|\hat{f}_n(x)-\mathbb{E}\hat{f}_n(x)|\leq C_1.
\end{equation}Next, the second term on the right hand side of inequality (\ref{42}) will be bounded. For all $x\in S$,
\begin{equation*}
\mathbb{E}(\hat{f}_n(x))=\mathbb{E}(\hat{f}_n(x)|d_H(\mathcal{X}_n,S)\leq\rho_n)\mathbb{P}(d_H(\mathcal{X}_n,S)\leq\rho_n)\end{equation*}
\begin{equation}\label{44}
+ \mathbb{E}(\hat{f}_n(x)|d_H(\mathcal{X}_n,S)>\rho_n)\mathbb{P}(d_H(\mathcal{X}_n,S)>\rho_n).
\end{equation}
Since $\{(x,y)\in S^2,\mbox{ }\|x-y\|\leq h_n\}$ is compact, the Lebesgue dominate convergence theorem entails that there exists $y_0\in S$ such that $\|x-y_0|\leq \rho_n$, and a sequence $y_k$ with $y_k$ tending to $y_0$, $\|y_k-y_0\|\leq \rho_n$, such that for $n$ large enough, with probability one,
$$\mathbb{E}(\hat{f}_n(x)|d_H(\mathcal{X}_n,S)\leq\rho_n)\leq \sup_{x\in S}\mathbb{E}\left(\limsup_{y\in S:\|x-y\|\leq\rho_n}f_n(y)\right)$$
$$=\sup_{x\in S}\mathbb{E}\left(\lim_{y_k\rightarrow y_0} f_n(y_k)  \right)=\sup_{x\in S}\lim_{y_k\rightarrow y_0}\mathbb{E}\left( f_n(y_k)  \right)\leq \sup_{x\in S} \sup_{y\in S:\|x-y\|\leq\rho_n}\mathbb{E}(f_n(y)).$$
Next, equation (\ref{40}) and Lipschitz continuity of $f$ allow to prove that
\begin{equation}\label{45}
\mathbb{E}(\hat{f}_n(x)|d_H(\mathcal{X}_n,S)\leq\rho_n)\leq \max_{y\in S:\|x-y\|\leq\rho_n}\{f(y)+k_fh_nc_K\}\leq f(x)+k_f\rho_n+k_f h_n c_K.
\end{equation}
With the same type of argument, we can ensure that
\begin{equation}\label{46}
\mathbb{E}(\hat{f}_n(x)|d_H(\mathcal{X}_n,S)\geq\rho_n)\leq \sup_{y\in S} \mathbb{E} f_n(y)\leq f_1+k_fh_nc_K.
\end{equation}From equations (\ref{44}), (\ref{45}), (\ref{46}) and (\ref{38}), we get
\begin{equation}\label{47}
\sup_{x\in S}\mathbb{E}(\hat{f}_n(x)-f(x)) \leq k_f\rho_n+k_fh_nc_K+ (f_1+k_fh_nc_K)C_Sn^{-2}.
\end{equation}

Taking $\epsilon_n=k_f\rho_n+k_fh_nc_K+(f_1+k_fh_nc_K)C_Sn^{-2}+C_1\left(\frac{n h_n^d}{-log(h_n)}\right)^{1/2}$ such that $log(n)\epsilon_n$ tends to zero. From equations (\ref{42}), (\ref{43}), (\ref{47}), we obtain that, with probability one, for $n$ large enough,
$$\max_{x\in S}(\hat{f}_n(x)-f(x))\leq\epsilon_n.$$
Then, for all $x\in S$, $\hat{f}_n(x)-f(x)\leq f(x)\epsilon_n/f_0$, and thus, 
$$\frac{\hat{f}_n(x)}{f(x)}\leq 1+\frac{\epsilon_n}{f_0},$$or equivalently,
$$\left(\frac{f(x)}{\hat{f}_n(x)}\right)\geq \left(1+\frac{\epsilon_n}{f_0}\right)^{-1/d}.$$
Finally, if $\epsilon_n^+=(1-(1+\epsilon_n/f_0)^{-1/d})\sim \epsilon_n/(df_0)$ then $\epsilon_n^+log(n)$ tends to zero. Therefore, 
$$\max_{x\in S}\left(\frac{f(x)}{\hat{f}_n(x)}\right)^{1/d}\geq 1-\epsilon_n^+, \mbox{ eventually almost surely}.$$This concludes the proof of (i).\\

In order to prove (ii), observe that
$$\min_{x\in\mathbb{R}^d}\hat{f}_n(x)\geq \min_{x\in \mathbb{R}^d}\mathbb{E}\hat{f}_n(x)-\max_{x\in \mathbb{R}^d}|\mathbb{E}\hat{f}_n(x)-\hat{f}_n(x)|.$$
Since we have already proved that $ \max_{x\in \mathbb{R}^d}|\mathbb{E}\hat{f}_n(x)-\hat{f}_n(x)|$ tends to zero almost surely, it only remains to check that $ \min_{x\in \mathbb{R}^d}\mathbb{E}\hat{f}_n(x)$ is bounded from below by a positive constant. From $  \min_{x\in \mathbb{R}^d}\mathbb{E}\hat{f}_n(x)= \min_{x\in \mathcal{X}_n}\mathbb{E}f_n(x)$ and (\ref{41}), we get
$$\min_{x\in\mathbb{R}^d}\hat{f}_n(x)\geq \min_{x\in S}\mathbb{E}f_n(x)\geq f_0c^{'}_k.$$

\end{proof}

$\vspace{.1cm} $\\
\emph{Acknowledgements.} The authors are grateful to Ignacio Munilla Rumbao for drawing his attention to the extent of occurrence estimation
problem and to Rosa M. Crujeiras for her useful and enriching comments. This work has been supported by Projects MTM2016-76969P and MTM2017-089422-P from the Ministry of Economy and Competitiveness and ERDF.

\end{document}